\documentclass[11pt]{article}

\usepackage[francais,english]{babel} 
\usepackage[utf8]{inputenc} 
\usepackage[T1]{fontenc} 
\usepackage{amsmath, amsthm}
\usepackage{amsfonts, amssymb}
\usepackage{lmodern} 
\usepackage[a4paper]{geometry} 
\usepackage{graphicx}
\usepackage{algorithm}
\usepackage{algpseudocode}
\usepackage{color}
\usepackage{dsfont}
\usepackage{caption}
\usepackage{url}
\usepackage{xcolor} 
\usepackage{enumitem}
\usepackage[colorlinks=true, linkcolor=blue,citecolor=orange]{hyperref} 
\usepackage{todonotes}

\definecolor{bleu2}{rgb}{0.8,0.83,1}
\definecolor{bleu3}{rgb}{0,0.2,1}
\definecolor{B2}{rgb}{0,0,0.7}
\definecolor{bleu-fonce}{rgb}{0,0,0.7}
\definecolor{bleu-nuit}{RGB}{1,25,147}
\definecolor{turquoise}{rgb}{0,1,1}
\definecolor{bleu}{rgb}{0,0,1}

\definecolor{vert}{rgb}{0,1,0}
\definecolor{vert2}{rgb}{0.86,1,0}
\definecolor{vert3}{rgb}{0.6,1,0.6}
\definecolor{vert-fonce}{rgb}{0,0.5,0}
\definecolor{vert-moy}{rgb}{0,0.7,0}
\definecolor{vert-clair}{rgb}{0.83,1,0.87}

\definecolor{J2}{rgb}{1,0.6,0}
\definecolor{jaune}{rgb}{1,1,0}
\definecolor{VI}{rgb}{0.5,0,0.5}
\definecolor{rose-vif}{RGB}{255,64,255}
\definecolor{ornage}{rgb}{1,0.75,0.5}
\definecolor{marron}{rgb}{0.5,0.25,0.1}
\definecolor{R2}{rgb}{1,0,0}
\definecolor{rose}{rgb}{1,0,1}
\definecolor{rouge-fonce}{rgb}{0.5,0,0}
\definecolor{rouge-moy}{rgb}{0.75,0,0}
\definecolor{rouge-1}{rgb}{1,0.65,0.65}
\definecolor{cadre-1}{rgb}{1,0.4,0.4}

\definecolor{violet}{rgb}{0.75,0.5,1}

\definecolor{W2}{rgb}{0.9,0.9,0.9}
\definecolor{gris-1}{rgb}{0.92,0.92,0.92}
\definecolor{gris-2}{rgb}{0.95,0.95,0.95}
\definecolor{gris-3}{rgb}{0.3,0.3,0.3}
\definecolor{gris-4}{rgb}{0.65,0.65,0.65}

\makeatletter
\renewcommand{\thefigure}{\ifnum \c@section>\z@ \thesection.\fi
 \@arabic\c@figure}
\@addtoreset{figure}{section}
\makeatother

\numberwithin{equation}{section} 

\theoremstyle{plain}
	\newtheorem{thm}{Theorem}[section]
	\newtheorem{prop}{Proposition}[section]
         \newtheorem{lem}{Lemma}[section]
         \newtheorem{assu}{Assumption}[section]

\theoremstyle{definition}

\theoremstyle{remark}
	\newtheorem{ex}{\textbf{Example}}[section]

\newcommand{\flot}{\phi}
\newcommand{\noyau}{Q}
\newcommand{\tauxdesaut}{l}

\newcommand{\tauxdecroissance}{r}

\newcommand{\HH}{\mathbf{H}}
\newcommand{\II}{\mathbf{I}}
\newcommand{\KK}{\mathbf{K}}
\newcommand{\LL}{\mathbf{L}}
\newcommand{\JJ}{\mathbf{J}}

\newcommand{\traitmarque}{\chi}

\newcommand{\Zmarque}{\mathcal{Z}}

\newcommand*{\Vmarque}[1]{\mathcal{V}_{#1}}

\newcommand{\flotmesure}{\Phi}
\newcommand{\tauxdesautmesure}{\lambda}
\newcommand{\Tauxdesautmesure}{\Lambda}
\newcommand{\noyaumesure}{\mathcal{Q}}

\newcommand{\HHmesure}{\mathbf{H}}
\newcommand{\IImesure}{\mathbf{I}}
\newcommand{\KKmesure}{\mathbf{K}}
\newcommand{\LLmesure}{\mathbf{L}}
\newcommand{\JJmesure}{\mathbf{J}}

\newcommand*{\Vmesure}[1]{\mathbb{V}_{#1}}

\newcommand{\traitaug}{\wt{X}}
\newcommand{\flotaug}{\wt{\Phi}}
\newcommand{\tauxdesautaug}{\wt{\lambda}}
\newcommand{\noyauaug}{\wt{\mathcal{Q}}}

\newcommand{\proba}{\mathbb{P}}
\newcommand{\Esp}{\mathbb{E}}  
\newcommand{\RR}{\mathbb{R}}

\newcommand{\NN}{\mathbb{N}}

\newcommand{\Exp}{\mathrm{e}}

\newcommand{\NNN}{\mathfrak{N}}
\newcommand{\MMM}{\mathfrak{M}}

\newcommand{\indic}{\mathds{1}}

\renewcommand{\epsilon}{\varepsilon}
\renewcommand{\leq}{\leqslant}
\renewcommand{\geq}{\geqslant}

\newcommand*{\wt}[1]{\widetilde{#1}}
\newcommand*{\ali}[1]{\[\begin{aligned}#1\end{aligned}\]}

\newcommand*{\fonction}[5]{ \begin{array}{ccccc}
#1 & : & #2 & \to & #3\\
 & & #4 & \mapsto & #5\\
\end{array} }

\begin{document}
\title{  \textbf{Optimal stopping for measure-valued piecewise deterministic Markov processes\footnote{This  work  was  partially  supported  by  R\'egion  Languedoc-Roussillon  and  FEDER  under
grant Chercheur(se)s d’Avenir, project PROMMECE}
} }
\author{Bertrand Cloez\footnote{MISTEA, INRA, Montpellier SupAgro, Univ Montpellier, Montpellier, France}, 
Benoîte de Saporta\footnote{IMAG, Univ Montpellier, CNRS, Montpellier, France}, 
Maud Joubaud\footnote{IMAG, Univ Montpellier, CNRS, Montpellier, France}}
\date{\today}

\maketitle

\begin{abstract}
This paper investigates the random horizon optimal stopping problem for measure-valued piecewise deterministic Markov processes (PDMPs).
This is motivated by population dynamics applications, when one wants to monitor some characteristics of the individuals in a small population. The population and its individual characteristics can be represented by a point measure.
We first define a PDMP on a space of locally finite measures.
Then we define a sequence of random horizon optimal stopping problems for such processes.
We prove that the value function of the problems can be obtained by iterating some dynamic programming operator.
Finally we prove on a simple counter-example that controlling the whole population is not equivalent to controlling a random lineage.
\end{abstract}

\newpage

\tableofcontents

\newpage
%
\section{Introduction}
%
Piecewise deterministic Markov processes (PDMPs) form a general class of non diffusion processes that was introduced by M. Davis in the 80's \cite{davis84,davis93}. 
Such processes have deterministic trajectories punctuated by random jumps. 
They belong to the family of hybrid processes with a discrete component called mode or regime interacting with a Euclidean component. 
PDMPs can model a wide area of phenomena from insurance and queuing problems \cite{davis93}, finance \cite{bauerle11}, reliability \cite{PDP15} to neuroscience \cite{GT14,RTW12}, population dynamics \cite{BL16, CDGMMY} 
and many other fields. 
In this paper we are especially interested in population dynamics applications. In this area, special cases of PDMPs include for instance growth-fragmentation processes 
for one or several interacting species \cite{CCF16, Cl17,Cos16,DHKR15}. In that case, commonly, the deterministic part is the growth process that may depend e.g. on the age of the individual, on its size, on the quantity of available nutriment and the jumps correspond to fragmentation or division (for cells), birth or death events, abrupt changes in the environment,\ldots

In Davis' original construction, PDMPs are defined on subsets of $\mathbb{R}^d$, for some dimension $d$ that may change 
when the process jumps. 
In this paper we are interested in extending the definition of PDMPs to measure-valued state spaces.
Infinite dimensional PDMPs have already been introduced in \cite{BR11} (see also \cite{GT14, RTW12}). In those papers, PDMPs take values in a separable Hilbert space and model spatio-temporal phenomena occurring on neuronal membranes. Our approach differs as we are interested in measure-valued PDMPs to deal with population dynamics models. Instead of modeling the dynamics of only a single individual by a finite-dimensional  PDMP, we aim at taking into account simultaneously the dynamics of all the individuals in the branching population when the population remains small and the stochastic approach is relevant and large scale approximations do not hold. Such a population can be represented by a point measure, 
hence the need to define measure-valued PDMPs.
The measure-valued process representation in population dynamics is used e.g. in  \cite{Ber06, FM04}, with fragmentation-type processes. It is a particular case of measure-valued PDMPs, with no deterministic dynamics between jumps and exponential distributions for the jump times. 

After constructing measure-valued PDMPs, we define a sequence of random horizon optimal stopping problems for measure-valued PDMPs and prove that the value functions can be obtained by iterating some dynamic programming operator. We also exhibit a sequence of $\epsilon$-optimal stopping times. Our approach is based on \cite{gugerli} that solved the optimal stopping problem for finite-dimensional PDMPs.

When dealing with a branching population, some important characteristics of the global population, e.g. laws of large numbers for functionals of the individuals, can be obtained by simply studying a suitably weighted random tagged lineage, by means of many-to-one formulas, see e.g. \cite{Ber06, DHKR15,Guy07}. Here, we prove that this property does not hold true for the optimal stopping problem. We provide a simple counter-example of cell division where stopping a suitably chosen tagged cell and the whole population yield different value functions.

The paper is organized as follows. In Section \ref{sec:defPDMP}, we construct measure-valued PDMPs. In Section \ref{sec:guger}, we state and solve the optimal stopping problem for such processes. Finally in Section \ref{sec:compa} we compare the value functions of the optimal stopping problems for the whole population and a tagged lineage.
%
\section{Construction of measure-valued PDMPs}
\label{sec:defPDMP}
%
This section is dedicated to the construction of piecewise deterministic Markov processes taking values in some measure space. 
Our construction of measure-valued PDMPs follows the same lines as in  \cite{davis84}: we first define the hybrid state space in which the process evolves, then we define the local characteristics giving the dynamics of trajectories between jumps, the jump times and the post-jump locations, and prove that a strong Markov process with such characteristics can be constructed. Finally we provide some toy example of such processes.
We start by setting some notation that will be used throughout the paper.
%
\subsection{Notation} 
Let $d$ be a positive integer. We denote by $\mathcal{B}$ the $\sigma$-field of Borel sets on  $\RR^d$ and $\mathcal{B}_b$ its subset of bounded Borel sets. More generally, For any topological space $E$,  we denote by $\mathcal{B}(E)$ its Borel $\sigma$-field, $B(E)$ its set of measurable bounded real-valued functions, $\overline{E}$ its closure, and $\partial E$ its boundary.

Let $\MMM$ be the set of locally finite measures on $(\RR^d,\mathcal{B})$ and $\NNN=\{\mu\in\MMM$;  $\forall B \in \mathcal{B}_b$, $\mu(B) \in \NN\}$ be the set of locally finite point measures. 
Note that any $\mu$ in $\mathfrak{N}$ can be expressed as a (possibly infinite) sum of Dirac distributions; the Dirac distribution with point mass at $x$, for $x$ in $\RR^d$, will be denoted by $\delta_x$.

Let $C_c(\RR^d)$ be the set of continuous real-valued functions with compact support on $\RR^d$. For any measure $\mu\in\MMM$ and function  $f\in C_c(\RR^d)$ set
\begin{equation*}
\mu f := \int_{\RR^d} f(x) \mu(dx).
\end{equation*}
We endow $\MMM$ with the vague topology. Recall that the vague convergence for a sequence of measures $(\mu_n)_n\subset \MMM$ to a measure $\mu\in\MMM$ is defined by
\begin{equation*}
\mu_nf\xrightarrow[n\rightarrow \infty]{} \mu f,\qquad  \forall f\in C_c(\RR^d).
\end{equation*}
We denote it $\mu_n \xrightarrow[n\rightarrow \infty]{v} \mu $. It is easily seen that  that $\NNN$ is a closed subset of $\MMM$ for the vague topology.\\

For any real numbers $a$ and $b$, $a \vee b$ and $a \wedge b $denote the maximum and minimum respectively between $a$ and $b$.
%
\subsection{State space of measure-valued PDMPs} 
%
Let $K$ be a finite set 
called modes or regimes space. For any mode $v \in K$, let $E_v$ be an open subset of $\MMM$, representing the state space in mode $v$. The global state space is then
\begin{equation*}
E=\{(v,\zeta) \in K \times \MMM \ | \ v\in K, \zeta \in E_v \}. 
\end{equation*}
We endow this set $E$ with the $\sigma$-field $\mathcal{E}$ generated by sets of the form $A=\{v\}\times A_v$ for all Borel sets $A_v\in \mathcal{B}(E_v)$. 
For instance, if one can consider the temporal evolution of a cell population characterized by state variables such as age, size, growth, maturity, protein content... Then the quantity $\zeta\in \NNN$ represents such state variables and the mode can be the experiment conditions. In the following example, we consider a simple case with only one mode and the variable state corresponds to the cell size.

\begin{ex} \label{ex:notation} 
Consider a population of 3 individuals at a given time. Their sizes are denoted by $(x_i)_{1\leq i \leq 3}$. This population is identified with the measure
\begin{equation*}
 \zeta=\sum_{i=1}^3 \delta_{x_i} \in \NNN. 
 \end{equation*}
It gives a complete view of the population: all the information is contained in $\zeta$. As time goes by, the number of individuals may increase or decrease leading to more or less terms in this Dirac sum representation but it will remain a measure in $\NNN$. This representation is then easier to manipulate than a changing-dimension vector. 
\end{ex} 

On the measure space $\MMM$, we introduce a particular metric $\rho$ for the vague topology in order to have a Polish space (i.e. a separable completely metrizable topological space). 
This property will be used in Section \ref{sec:loc}, for the explicit construction of the stochastic process. 
As shown in \cite[Appendix]{Ka86}, a suitable choice for $\rho$ is constructed as follows. 
Let $\mathcal{C}$ be a countable basis of open bounded subsets in $\RR^d$ closed under finite unions. For all $C$ in $\mathcal{C}$, it exists a sequence $(C_n)_n$ in $\mathcal{B}_b$ and an increasing sequence $(f_{C,n})_n$ in $C_c(\RR^d)$ such that
\begin{equation*}
f_{C,n} \xrightarrow[n\rightarrow \infty]{} \indic_C \quad\text{ and }\quad\indic_{C_n}\leq f_{C,n}\leq\indic_{C}.
\end{equation*}
Since $\mathcal{C}$ is countable, the set $\{f_{C,n} ~|~ C\in\mathcal{C}$,  $n\in\NN\}$ is also countable, and we then number $f_1,f_2,\ldots$ those functions. Any measure $\mu$ is completely determined by the set $\{ \mu f_k, k\in\NN\}$. Now, for all $\mu$ and $\mu'$ in $\MMM$, we define the distance $\rho$ by
\begin{equation*}
 \rho(\mu,\mu'):= \sum_{k\geq1}\frac{1}{2^k}\left( 1-\exp(-|\mu f_k-\mu' f_k|)\right).
 \end{equation*}

From this metric on $\MMM$, we define a metric on $E$ related to its hybrid structure: any two points in $E$ with different modes must be arbitrarily far away from each other.
For all $x=(v,\zeta)$ and $x'=(v',\zeta')$ in $E$ set 
\begin{equation*}
 \rho_0(x,x'):=
 \begin{cases} { \displaystyle 
 \frac{2}{\pi} \arctan(\rho(\zeta,\zeta'))} &\text{ if } v= v' ,\\
1 &\text{ otherwise}.
\end{cases}
 \end{equation*}
Thus $\rho_0(x,x')$ equals 1 if and only if $x$ and $x'$ have different modes.
With this metric, a sequence $(x_n)_n=(v_n,\zeta_n)_n$ converges to $x=(v,\zeta)$ in $(E,\rho_0)$ if and only if it exists some $m$ in $\NN$ such that:
\begin{equation}
\label{eq:equiv-cv}
\begin{cases}v_n
=v \text{ for }n\geq m, \\ 
\zeta_{m+k} \xrightarrow[k\rightarrow \infty]{} \zeta  \text{ in } (E_m,\rho).
\end{cases}
\end{equation}
%
We thus denote 
\begin{equation*}
\overline{E}=\{(v,\zeta) \in K \times \MMM \ | \ v\in K, \zeta \in \overline{E}_v \},
\end{equation*}
the closure of $E$ for the distance $\rho_0$. The following statement is then straightforward.
%
\begin{lem}
The metric space  $(\overline{E},\rho_0)$ is a Polish space.
\end{lem}
\subsection{Construction of measure-valued PDMPs}
\label{sec:loc}
We now introduce the three local characteristics of the PDMP specifying the dynamics of trajectories between jumps, the jump times and the post-jump locations.
\begin{itemize}
\item The flow $\flotmesure$ is defined by $(x,t) \mapsto \flotmesure(x,t)=(v,\flotmesure_v(\zeta,t))$ for all $x=(v,\zeta)$ in $E$ and non-negative $t$, where the functions $\Phi_v:\mathfrak{M}\times \RR\rightarrow \mathfrak{M}$ are continuous 
and 
have a semi-group property: for all $s,t \geq0$, we have $\flotmesure_v(\cdot,t+s)=\flotmesure_v(\flotmesure_v(\cdot,s),t)$. 
The flow describes the deterministic trajectory of the process between jumps. Let
\begin{equation*}
t^*(x)=\inf\{t>0,~\flotmesure_v(t,\zeta)\in \partial E_v\}
\end{equation*}
be the deterministic time the flow takes to reach the boundary of the domain starting from $x=(v,\zeta)\in E$, with the usual convention $\inf\emptyset=+\infty$. An infinite exit time $t^*$ means that the process cannot reach the boundary in finite time.
 
 \item The jump intensity $\tauxdesautmesure:E \rightarrow \RR_+$ is a mesurable function, with a local integrability property: for all $x=(v,\zeta)\in E$ there exists some $\varepsilon>0$ such that
\begin{equation*}
\int_0^\epsilon \tauxdesautmesure(v,\flotmesure_v(\zeta,s)) ds <\infty.
\end{equation*}
It determines the frequency of the jumps.

\item The Markov kernel $\noyaumesure:  \overline{E} \times \mathcal{B}(E) \rightarrow [0,1]$ selects the post-jump locations. It has the following property: 
\begin{equation*}
 \forall x\in E,~\noyau(x,\{x\})=0,
 \end{equation*}
meaning that he process cannot have a no-move jump.
\end{itemize}

From these local characteristics, one can construct a stochastic process similarly to \cite[Section 24]{davis93} as follows.
Let $(\Omega,\mathcal{F},\proba)$ be the canonical space for a sequence $(U_n)_{n\geq 1}$ of independent random variables with uniform distribution on $[0,1]$. The sample path of an $E$-valued PDMP $(X_t(\omega))_{t\geq 0}$ starting from a fixed initial point $x=(v,\zeta)\in E$ and for some $\omega\in \Omega$ is defined iteratively. Let 
\begin{equation} \label{eq:loiT1}
F(x,t)=\indic_{\{t<t^*(x)\}} \exp\left(-\int_0^t \tauxdesautmesure(v,\flotmesure_v(\zeta,s))ds\right)
\end{equation}
for $(x,t)=(v,\zeta,t)\in E\times\mathbb{R}_+$ and $ \Psi_1$ be the function from $E\times [0,1]$ onto $\mathbb{R}_+$ defined by
\begin{equation*} 
 \Psi_1(x,u)=\inf\{t\geq 0;\  F(x,t)\leq u\},
\end{equation*}
and define $S_1(\omega)=T_1(\omega)= \Psi_1(x,U_1(\omega))$ the first jump time of the process. Thus $F(x,\cdot)$ is the survivor function of $T_1$.\\

As the spaces $\MMM$ and $\overline{E}$ are Polish, one can use \cite[p. 6]{Ko11} to obtain that $\{\noyaumesure(x,\cdot) \}_{x\in \overline{E}}$ is a collection of probability measures on $\overline{E}$, with a measurable dependence on the parameter $x$ in $\overline{E}$. Then there is a measurable function $\Psi_2:E\times [0,1]\rightarrow E$ such that the distribution of $\Psi_2(y,U)$ is $\noyaumesure(y,\cdot)$ for any  random variable $U$ with uniform distribution on $[0,1]$. Hence, one can then define 
\begin{align*} 
X_t(\omega)&=\flotmesure_v(\zeta,t),\quad \text{ for } 0\leq t<T_1, \\
X_{T_1}(\omega)&=\Psi_2\big((v,\flotmesure_v(\zeta,T_1(\omega))),U_2(\omega)\big).
\end{align*}
Hence the trajectory $X_t$ follows the deterministic flow starting from $X_0=x$ until the first jump time $T_1$. At $T_1$ a new location $Z_1=X_{T_1}$ is drawn according to the Markov kernel $\noyaumesure$. Namely, the law $\mathcal{L}(Z_1 \ | \ T_1)$ of  $Z_1$ conditionally on $T_1$ is 
\begin{equation} \label{eq:loiZ1}
\mathcal{L}(Z_1 \ | \ T_1)=   \noyaumesure(\flotmesure(x,S_1),\cdot),
\end{equation}

The process now restarts from $X_{T_1}$ following the same steps. Define
\begin{equation*} 
S_2(\omega)= \Psi_1(X_{T_1}(\omega),U_3(\omega)),\quad T_2(\omega)=S_1(\omega)+S_2(\omega),
\end{equation*}
and set 
\begin{align*} 
X_t(\omega)&=\flotmesure(X_{T_1}(\omega),t-T_1(\omega)),\quad \text{ for } T_1(\omega)\leq t<T_2(\omega),\\
X_{T_2}(\omega)&=\Psi_2\big((v,\flotmesure(X_{T_1}(\omega),S_2(\omega))),U_4(\omega)\big),
\end{align*}
and so on. In order to avoid explosion issues, the following assumption will hold throughout the paper.
\begin{assu} \label{ass:T}
For all $(x,t)$ in $E\times \RR_+$, $\Esp_x[\sum_{n=1}^{\infty} \indic_{(T_n<t})]< \infty$.
\end{assu}
Hence the trajectories of $(X_t)$ are well defined for all $t\geq 0$.\\

The positive random variables $S_1,S_2,\ldots$ are the times between two consecutive jumps or inter-jump times. For notational convenience, we set $T_0=S_0=0$ and $T_n:=\sum_{i=1}^n S_n$ the $n^{th}$ jump time of the process. Note that we have $T_1=S_1$ and $S_{n+1}=T_{n+1}-T_n$. The sequence $(Z_n)_n$ with $Z_n=X_{T_n}$ describes the post jump locations of the process. By construction, all the randomness of the continuous-time process $(X_t)_{t\geq 0}$ is contained in the discrete-time process $(Z_n,S_n)_{n\geq 0}$.
%
\subsection{Structure of stopping times and Markov property}
\label{sec:markov}
%
The aim of this section is to prove that the special structure of stopping times for finite dimensional PDMPs given in \cite[Theorem A2.3]{davis93} still holds in the measure space context. This yields the Markov property and will be important for the study of the optimal stopping problem in Section \ref{sec:guger}.\\

Let $D_E[0,+\infty)$ be the set of unctions on $\mathbb{R}_+$ with values in $E$ that are right-continuous with left limits. Denote by $\tilde X_t$ the coordinate function $\tilde X_t(f)=f(t)$ for $f\in D_E[0,+\infty)$. Let $(\mathcal{F}_t^0)_{t\geq 0}$ denote the natural filtration of $(\tilde x_t)$ and $\mathcal F^0=\vee_{t\geq 0}\mathcal F^0_t$. Under Assumption \ref{ass:T}, for each starting point $x\in E$, the construction in the previous section defines a measurable mapping $\Psi_x$ from $\Omega$ onto $D_E[0,+\infty)$ such that $\tilde X_t(\Psi_x(\omega))=X_t(\omega)$. Let $\mathbb P_x$ denote the image measure  of $\mathbb P$ by $\Psi_x$. This defines a family of measure $(\mathbb P_x)_{x\in E}$ on $D_E[0,+\infty)$. In the sequel, we identify $\tilde X_t$ and $X_t$. 
For any probability measure $\nu$ on $E$ define the measure $\mathbb P_\nu$ on $(D_E[0,+\infty),\mathcal F^0)$ by $\mathbb P_\nu(\cdots)=\int_E\mathbb P_x(\cdot)\nu(dx)$. Now let $\mathcal{F}_t^\nu$ be the completion of $\mathcal{F}_t^0$ with all $\mathbb P_\nu$-null sets of $\mathcal{F}^0$ and define $\mathcal{F}_t=\cap_{\nu\in\mathcal{P}(E)}\mathcal{F}_t^\nu$, where $\mathcal{P}(E)$ is the set of probability measures on $E$.\\

We can now state a crucial result on the structure of stopping-times for our process.
%
\begin{thm}\label{th:stop time}
A non-negative random variable $\tau$ is a $(\mathcal{F}_t)_{t\geq 0}$ stopping-time if and only if there exists a sequence $(R_n)_{n\in \mathbb N}$ of non-negative $(\mathcal{F}_{T_n})_{n\in \mathbb N}$-adapted random variables such that
\begin{equation*}
\tau=\sum_{n=1}^\infty R_{n-1}\wedge S_n.
\end{equation*}
\end{thm}
%
\begin{proof} 
The proof follows the same lines as in \cite[Section 1.7]{PDP15}. 
It is based on Theorem A2.3 in \cite{davis93} that is valid for any right-continuous piecewise constant process taking values in a Borel set and the one-to-one correspondence between our process $(X_t)$ and the right-continuous piecewise constant $E\times\mathbb N$-valued process $(\eta_t)$ defined by 
\begin{align*}
\eta_t=(X_0,0), \ t<T_1,\quad \eta_t=(X_{T_n},n),\ T_n\leq t< T_{n+1},
\end{align*}
see \cite[(25.1)]{davis93}.
\end{proof}
%
\begin{thm}
The process $(X_t)_{t\geq 0}$ on $(D_E[0,+\infty), \mathcal{F}, (\mathcal{F}_t), (\mathbb P_x)_{x\in E})$ is a strong Markov process.
\end{thm}
%
\begin{proof} 
The strong Markov property is proved in the same way that in \cite[Section 25]{davis93}, using Theorem \ref{th:stop time} and the one-to-one correspondence with the piecewise constant process $(\eta_t)$ defined above.
\end{proof}
 %
\subsection{Toy example}
\label{ex}
%
We now develop Example~\ref{ex:notation} into a more generic model for some cells population. Typically, cells grow and divide into two daughter-cells that start growing and then divide and so on. 
%
\subsubsection{Single cell model}
\label{sect:ex1}
%
We first define a single-cell model that follows standard final dimensional PDMP dynamics, where one randomly selects a single daughter cell at each division. This model was further studied in \cite{DHKR15}.
%
We consider a model with a single mode, hence the state space is simply $\mathbb{R}_+$.
For $\xi$ in $\RR_+$ and $t>0$, the size at time $t$ for a cell with initial size $\xi$ is given by 
$\flot(\xi,t)=\xi\exp(\tauxdecroissance t)$, where $r>0$ is the common growth rate for all the cells. 
The jump intensity giving the division dates is given by $\tauxdesaut(\xi)=\xi^{\alpha}$, for some positive number $\alpha$. 
This is simplified model consistent with the statistical evidence that division is triggered by the cell size rather than its age, see \cite{krell}.
The jump kernel is simply a division by two $Q(x,A)=\delta_{x/2}(A)$ for any Borel subset $A$ of $\mathbb{R}_+$.
%
\subsubsection{Population model}
\label{sect:extous}
We now consider the previous cell growth-division dynamics but for the whole population instead of a single randomly selected cell.
Again we consider a single mode so that he space $E$ is simply $\NNN$.

Let $\zeta=\sum_{i=1}^n \delta_{x_i}$  be in $\NNN$ be an initial state of $n$ cells with respective sizes $(x_i)_{1\leq i\leq n}$.
Each cell grows following the previous dynamics, so that globally the flow is 
\[ \fonction{{\flotmesure}}{\mathfrak{N}\times\RR_+}{\mathfrak{N}}{(\zeta,t)}{\sum_{i=1}^n \delta_{\flot(x_i,t)}.}\]
The first jump time corresponds to the first split ie the minimum between $n$ exponentially distributed random variables. 
Thus, the jump intensity is:
\[ \begin{array}{ccccc}
{\tauxdesautmesure} & : & \mathfrak{N} & \to & \RR \\
 & & \zeta& \mapsto & \sum_{i=1}^n l(x_i).\\
\end{array}  \]
Given that the $j^{th}$ cell is the one which split, the post-jump location is 
\[Z_1=\sum_{\substack{ i=1\\i\neq j}}^n \delta_{x_i}+2 \delta_{x_j/2}, \]
where celle number $j$ was removed and two new cells are added with half the size of cell $j$. For notational convenience, we set
\[\zeta_{(j)}:=\zeta-\delta_{x_j}+2 \delta_{x_j/2}. \]
Thus, for $A$ in $\mathcal{E}$, the Markov jump kernel of the PDMP is given by 
\[ {\noyaumesure}(\zeta,A)=\sum_{j=1}^n \frac{l(x_j)}{\sum_{i=1}^n l(x_i)} \indic_A(\zeta_{(j)}). \]
%
\subsubsection{Time-augmented population model}  \label{sec:aug}
Sometimes, it is convenient to add time in the state variable of a PDMP, for instance when one wants to use time-dependent jump intensity, or trigger some jump when a certain lapse of time has passed, or to study control problems with time-dependent reward functions. In \cite[Section 31]{davis93} Davis proves that for finite dimensional PDMPs, the time-augmented process is still a PDMP. The same property holds in the framework of our toy example. \\ 
We consider the time-augmented process $\traitaug_t:$ defined as the PDMP on $\NNN$ starting from the initial state $\tilde\zeta=\sum_{i=1}^n \delta_{(x_i,0)}$. The flow is now given by
\[ \fonction{{\flotaug}}{\mathfrak{N}\times\RR_+}{\mathfrak{N}}{(\tilde\zeta,t)}{\sum_{i=1}^n \delta_{(\flot(x_i,t), u+t)},}\]
for any $\tilde\zeta$ in $\NNN$  of the form $\tilde\zeta=\sum_{i=1}^n \delta_{(x_i,u)}$. The 
jump intensity is
\[ \begin{array}{ccccc}
{\tauxdesautaug} & : & \mathfrak{N} & \to & \RR \\
 & & \tilde\zeta=\sum_{i=1}^n \delta_{(x_i,u)}& \mapsto & \sum_{i=1}^n l(x_i),\\
\end{array}  \]
and the jump kernel is simply
\[ \noyauaug(\sum_{i=1}^n \delta_{(x_i,u)},A)=\sum_{j=1}^n \frac{l(x_j)}{\sum_{i=1}^n l(x_i)} \indic_A(\sum_{i=1}^n \delta_{(x_i,u)}-\delta_{(x_j,u)}+2\delta_{(x_j/2,u)}).\]
 %
 We will further study this toy example in Section \ref{sec:compa} to prove that controlling the whole population is not equivalent to controlling a suitably chosen random lineage.
 %
\section{Optimal stopping problem} \label{sec:guger}
%
We now turn to the main aim of this paper: defining and solving the optimal stopping problem for measure-valued PDMPs.
Roughly speaking, one wants to stop the process at the best time in order to maximize some reward depending on the state of the process when stopped.
More precisely, let $(X_t)_{t\geq 0}$ be an $E$-valued PDMP on $(D_E[0,+\infty),\mathcal{F},(\mathcal{F}_t)_{t\geq 0},(\proba_x)_{x\in E})$ and  $g\in B(E)$ be some non-negative reward function.
Denote by $\mathcal{M}$ the set of stopping-times with respect to the filtration $(\mathcal{F}_t)_{t\geq 0}$ and
for all positive integer $N$, let  $\mathcal{M}_N$ be the set of stopping-times bounded by the $N^{th}$ time jump $T_N$ of the PDMP
\[ \mathcal{M}_N=\{\tau\in \mathcal{M};\ \tau \leq T_N\}. \]
For all $x\in E$, set 
 \begin{equation}  \label{eq:reward} 
 \mathbb{V}(x):=\sup_{\tau\in\mathcal{M}_N}\mathbb{E}[g(X_\tau)\ |\ X_0=x].   
 \end{equation}
 Thus $ \mathbb{V}(x)$ is the best possible (average) performance when stopping a PDMP starting from $X_0=x$ before its $N$-th jump. Function $\mathbb{V}$ is called the \emph{value function} of the optimal stopping problem. Solving an optimal stopping problem consists in characterizing the value function as the solution of some recursive equations called \emph{dynamic programming} equations and exhibiting a family of \emph{$\epsilon$-optimal stopping times} $\tau_\epsilon\in\mathcal{M}_N$ such that
 \begin{equation*} 
 \mathbb{V}(x)-\epsilon\leq \mathbb{E}[g(X_{\tau_\epsilon})\ |\ X_0=x]\leq \mathbb{V}(x).
 \end{equation*}
%
In this section, we first define some suitable dynamic programming operators and a family of stopping times. Then we prove that the value functions can be constructed by iteration of the dynamic programming operators and that the stopping times are $\epsilon$-optimal. This section is inspired from the study of the optimal stopping problem for finite (fixed) dimension PDMPs derived in \cite{gugerli}.
 %
\subsection{Dynamic programming operators} 
%
We start with some additional notation and assumption. 
For $w$ in $B(E)$, $x$ in $E$ and $l$ a measurable real-valued function on $E$, we denote in short
 \[l\noyaumesure w (x) := l(x)\times \noyaumesure w(x)=l(x)\int_E w(y) \noyaumesure (x,dy). \]
The following assumption is made for simplicity reasons. It is satisfied in most real-life examples for instance when monitoring a population until some finite horizon time. 
\begin{assu} The exit time $t^*$ is in  $B(E)$.
\end{assu}

We now define some operators on $B(E)$.
Let $\HHmesure$ and $\IImesure$ be the operators from  $B(E)$ onto  $B(E\times \RR_ +)$ defined for all $w$ in $B(E)$, $x$ in $E$ and $t\in\mathbb{R}_+$, by
 \begin{align*}
\HHmesure w(x,t)&= w\left(\flotmesure(x,t\wedge t^*(x))\right)~\Exp^{-\Tauxdesautmesure(x,t\wedge t^*(x))},\\  
{\IImesure}{w}(x,t)&=\int_0^{t\wedge t^*(x)}\tauxdesautmesure\noyaumesure w(\flotmesure(x,s))~\Exp^{-\Tauxdesautmesure(x,s)}\mathrm{d}s,
 \end{align*}
 where 
 \begin{align*}
\Tauxdesautmesure(x,t)=\int_0^t \lambda (\flotmesure(x,s))ds.
 \end{align*}
We also introduce operator $\KKmesure$ from $B(E)$ onto $B(E)$, defined for all $w$ in $B(E)$ and $x$ in $E$ as
\[ {\KKmesure}{w}(x)= \int_0^{t^*(x)}\tauxdesautmesure\noyaumesure w(\flotmesure(x,s))~\Exp^{-\Tauxdesautmesure(x,s)}\mathrm{d}s + \noyaumesure w (\flotmesure(x,t^*(x))) \Exp^{-\Tauxdesautmesure(x,t^*(x))}. \]
It is straightforward to see that these operators can be expressed as expectations involving the embedded Markov chain $(Z_n,S_n)$ defined in Section \ref{sec:loc}.
%
\begin{prop}
For all $w$ in $B(E)$, $x$ in $E$ and $t\geq 0$ one has:
\ali{\HH w(x,t) &= \Esp_x\left[w(X_{t\wedge t^*(x)})~\indic_{S_1> t\wedge t^*(x)}\right]=w\left(\flotmesure(x,t\wedge t^*(x))\right) \proba_{x}\left(S_1>t\wedge t^*(x)\right),\\
\II w(x,t) &= \Esp_x\left[ w(Z_1)~\indic_{S_1 \leq t\wedge t^*(x)} \right],\\
\KK w(x)&=\Esp_x[w(Z_1)].
}
\end{prop}
Finally we denote by $\JJ$ and $\LL$ the operators from $B(E)$ onto $B(E\times\mathbb R_+)$ and $B(E)$ respectively defined for all $w$ in $B(E)$, $x$ in $E$ and $t\in\mathbb{R}_+$ by
 \begin{align*}
\JJ(w,g)(x,t)&= \HH g(x,t)+\II w(x,t),\\
\LL(w,g)(x)&= \sup_{t\geq 0} \left\{ \JJ(w,g)(x,t) \right \} \vee \KK w (x),
 \end{align*}
%
where $g$ is the reward function of the optimal stopping problem.
Roughly speaking, operator $\LL$ represents the best compromise between stopping at the best location along the deterministic trajectory following the flow ($ \sup_{t\geq 0}  \JJ$ part) or waiting until the next jump ($\KK$ part).
%
\subsection{Family of stopping-times} 
%
Now we introduce a family of random variables and prove they are stopping times. They will be candidates $\epsilon$-optimal stopping times for our optimal stopping problem.
For all $x\in E$, $n\in\mathbb{N}$ and $\epsilon>0$ set
$${\displaystyle 
r_{n,\epsilon}(x)=\begin{cases} t^*(x)\text{ if }\KKmesure\Vmesure{n}(x)> \sup_{t>0}\JJmesure(\Vmesure{n},g)(x,t), \\ 
\inf\{s\geq0~;~\JJmesure(\Vmesure{n},g)(x,s)\geq\sup_{t>0}\JJmesure(\Vmesure{n},g)(x,t)-\epsilon \} &\text{otherwise.} \end{cases}
}$$

For all $\epsilon>0$ and $n\geq 2$ also set 
\begin{align*}
R_{1,\epsilon}&=r_{0,\epsilon}(Z_0),\\
R_{n,0}^\epsilon&=r_{n-1,\epsilon/2}(Z_0),\\
R_{n,k}^\epsilon&=r_{n-k-1,\epsilon/2^{k}}(Z_k)\indic_{(R_{n,k-1}^\epsilon\geq S_k)}, \quad 1\leq k\leq n-1.
\end{align*}
Then $R_{n,k}^\epsilon$ is clearly $\mathcal F_{T_k}$-measurable for $0\leq k\leq n-1$.
Finaly, define
$S_{1,\epsilon}:=r_{0,\epsilon}(Z_0)\wedge T_1=R_{1,0}^\epsilon\wedge S_1$, and by iteration,
$${\displaystyle 
S_{n,\epsilon}:=
\begin{cases} 
R_{n,0}^{\epsilon} &T_1>R_{n,0}^{\epsilon} ,\\ 
T_1+\theta(T_1)S_{n-1,\epsilon/2} & T_1\leq R_{n,0}^{\epsilon},\end{cases}
}$$
where $\theta(t)$ is the shift operator with lag $t$ on $D_E[0,+\infty)$, namely for $f\in D_E[0,+\infty)$, $\theta(t)f(\cdot)=f(t+\cdot)$. 

In order to prove that the $S_{n,\epsilon}$ are stopping-times in $\mathcal M_n$, we first study the effect of the shift operator on $R_{n,k}^\epsilon$.
\begin{lem} \label{lem:stopp time} 
For all $\epsilon>0$ and $n\geq 2$ and $1\leq k\leq n-1$, on the set $(T_1\leq S_{n,2\epsilon})$, one has
\begin{align*}
R_{n,k}^{2\epsilon}=\theta(T_1)(R_{n-1,k-1}^\epsilon).
\end{align*}
\end{lem} 
%
\begin{proof} 
For $n=2$, by definition, one has $R_{1,0}^\epsilon=r_{0,\epsilon}(Z_0)$ hence $\theta(T_1)(R_{1,0}^\epsilon)=r_{0,\epsilon}(Z_1)$ and $R_{2,1}^{2\epsilon}=r_{0,\epsilon}(Z_1)$ on $(T_1\leq S_{2,2\epsilon})=(T_1\leq R_{2,0}^\epsilon)$. Hence the result holds.\\
For $n\geq 3$ we prove the result by induction on $k$ using similar arguments and the fact that $(T_1\leq S_{n,\epsilon})=(R_{n,0}^{2\epsilon}\geq T_1)$.
\end{proof}

We now prove that $S_{n,\epsilon}$ is a stopping-time using the characterization of Theorem \ref{th:stop time}.

\begin{lem} \label{lem:stopp time} 
For all $\epsilon>0$ and $n\in\mathbb N^*$ one has 
\begin{align*}
S_{n,\epsilon}=\sum_{k=1}^n R_{n,k-1}^\epsilon\wedge S_k.
\end{align*}
In particular, $S_{n,\epsilon}$ is a stopping time and $S_{n,\epsilon}\leq T_n$.
\end{lem} 
%
\begin{proof} 
We proceed by induction on $n$. 
For $n=1$, by definition $S_{1,\epsilon}=r_{0,\epsilon}(Z_0)\wedge T_1=R_{1,k-1}^\epsilon\wedge S_1$ and the result is true.\\
Suppose the result holds for $n-1$.
From the definition, on the event $(r_{n-1,\epsilon/2}(Z_0)=R_{n,0}^\epsilon<S_1)$, one has $S_{n,\epsilon}=R_{n,0}^\epsilon$ and $R_{n,1}^\epsilon=0<S_2$, $R_{n,2}^\epsilon=0<S_3$, and so on, so that $\sum_{k=1}^n R_{n,k-1}^\epsilon\wedge S_k=R_{n,0}^\epsilon$ and the result is valid.\\
On the event $(R_{n,0}^\epsilon\geq S_1)$, by definition $S_{n,\epsilon}=T_1+\theta(T_1)S_{n-1,\epsilon/2}$. 
Now, the induction hypothesis and the previous lemma yield
\begin{align*}
T_1+\theta(T_1)S_{n-1,\epsilon/2}
&=S_1+\sum_{k=1}^{n-1} R_{n,k}^{\epsilon}\wedge S_{k+1}
=\sum_{k=1}^n R_{n,k-1}^\epsilon\wedge S_k,
\end{align*}
on $(R_{n,0}^\epsilon\geq S_1)$. Hence the result.
\end{proof}

%
\subsection{Characterization of the value function} 
%
We can now propose an iterative construction of value functions by iterating operator $\LL$. For all $x\in E$, set
%
\begin{equation} 
\begin{aligned}  
\begin{cases} \Vmesure{0}(x)&=g(x), \\
\Vmesure{n}(x)&=\LLmesure(\Vmesure{n-1},g)(x) \text{ for all }n \geq 1.
\end{cases}
\end{aligned}  
\label{eq:rec:gu}
\end{equation}
 Clearly the functions $\Vmesure{n}$ are in $B(E)$.
We can now state and prove our main result, namely that the value function $\mathbb{V}$ of the optimal stopping problem (\ref{eq:reward}) equals the $N$-th iterate $\Vmesure{N}$ and that $S_{N,\epsilon}$ is an $\epsilon$-optimal stopping time.
%
\begin{thm} \label{thm:rec} 
Let $x$ be in $E$ and $n\in \NN$. Then one has, for all $\epsilon>0$, $S_{n,\epsilon}$ is a stopping-time in $\mathcal{M}_n$ and
\begin{eqnarray}  \label{gu:a} 
\Vmesure{n}(x)=\sup_{S\in \mathcal{M}_n}\Esp_{x}[g(X_S)],&    \\ 
\Esp_{x}[g(X_{S_{n,\epsilon}})]\geq \Vmesure{n}(x)-\epsilon. \label{gu:b}
\end{eqnarray}
\end{thm}
%
\begin{proof} 
By an induction argument we will prove Equation \eqref{gu:b} and the following inequality
\begin{equation} \label{gu:aa} 
\forall S \in \mathcal{M},\quad \forall n\in\NN, \qquad\Esp_{x}[g(X_{S\wedge T_n})] \leq \Vmesure{n}(x).
\end{equation}%
These two equations imply \eqref{gu:a}. Indeed, for all $S$ in $\mathcal{M}_{\infty}$, we have $S\wedge T_n \in \mathcal{M}_n$ then for all  $S$ in $\mathcal{M}_{\infty}$, $n$ in $\NN$ and $x$ in $E$, \eqref{gu:b} and \eqref{gu:aa} yield
$$
\Vmesure{n}(x)-\epsilon \leq  \Esp_{x}[g(X_{S_{n,\epsilon}})]  \leq \sup_{S\in \mathcal{M}_n} \Esp_{x}[g(X_{S})] = \sup_{S\in \mathcal{M}_n} \Esp_{x}[g(X_{S\wedge T_n})] \leq \Vmesure{n}(x),
$$
which is valid for all positive $\epsilon$, hence the result.\\

It remains now to prove \eqref{gu:aa} and \eqref{gu:b} by induction.\\

The case $n=1$ is based on Theorem \ref{th:stop time}. Indeed,  we obtain $ \Esp_{x}[g(X_{S\wedge T_1})] =\Esp_{x}[g(X_{R_0\wedge T_1})]$, for some $\mathcal{F}_0$-measurable random variable $R_0$. From this, we deduce 
\begin{align*}
 \Esp_{x}[g(X_{S\wedge T_1})] 
 &=\Esp_{x}[g(X_{R_0})\indic_{T_1 > R_0}] +\Esp_{x}[g(Z_1)\indic_{T_1\leq R_0}]\\
 &= \HHmesure g (x,R_0)+ \IImesure g (x,R_0) = \JJmesure(g,g)(x,R_0)\\
 &\leq \sup_{t\geq 0} \left\{ \JJmesure(g,g)(x,t) \right\}\leq  \Vmesure{1}(x);
 \end{align*}
which proves \eqref{gu:aa} for $n=1$. To prove \eqref{gu:b}, we distinguish two cases:\\
$\bullet$ if $ \KKmesure g (x) > \sup_{t\geq 0}\JJmesure(g,g)(x,t) $, then $\Vmesure{1}(x)=\KKmesure g(x)=\Esp_{x}[g(Z_1)]$ and also 
$$
S_{1,\epsilon}=t^*(x) \land T_1=T_1.
$$
Thus  $ \Esp_{x}[g(X_{S_{1,\epsilon}})]=\Esp_{x}[g(Z_1)]=\Vmesure{1}(x)\geq \Vmesure{1}(x)-\epsilon$. \\
$\bullet$ otherwise, $\Vmesure{1}(x)=\sup_{t\geq 0}\JJmesure(g,g)(x,t)$ and, by definition of $\JJmesure$, 
\begin{align*}
 \Esp_{x}[g(X_{S_{1,\epsilon}})]
 &=\Esp_{x}[g(X_{r_{0,\epsilon}(x)\land T_1})]=\Esp_{x}[g(X_{r_{0,\epsilon}(x)})\indic_{T_1 \geq r_{0,\epsilon}(x)}] +\Esp_{x}[g(Z_1)\indic_{T_1\leq r_{0,\epsilon}(x)}]\\
 &=\JJmesure(g,g)(x,r_{0,\epsilon}(x)).
 \end{align*}
We then deduce, by definition of $r_{0,\epsilon}$: 
\[ \Esp_{x}[g(X_{S_{1,\epsilon}})] \geq \sup_{t\geq 0}\{ \JJmesure(g,g)(x,t)\} -\epsilon=\Vmesure{1}(\mu)-\epsilon. \]
This completes the proof of  \eqref{gu:b} for $n=1$.\\

Now, let $N\geq 1$, suppose that \eqref{gu:aa} and \eqref{gu:b} hold true for all $n\leq N$ and prove that they hold for $n=N+1$.

Begin by proving \eqref{gu:aa}. Again, this is based on Theorem \ref{th:stop time}. As $S$ is a stopping-time, it can be decomposed as 
\begin{equation*}
S=\sum_{n=1}^\infty R_{n-1}\wedge S_n.
\end{equation*}
where $(R_n)_{n\in \mathbb N}$ are non-negative $(\mathcal{F}_{T_n})_{n\in \mathbb N}$-adapted random variables. In particular, on $(S\geq T_1)$ one has 
\begin{equation*}
S=S_1+\sum_{n=2}^\infty R_{n-1}\wedge S_n=T_1+\theta(T_1)(S'),
\end{equation*}
for some stopping-time $S'$.
Thus one has
\begin{align} 
\Esp_{x}[g(X_{S\wedge T_{n+1}})]
&= \Esp_{x}[g(X_{S\wedge T_{1}})\indic_{S<T_1}]+ \Esp_{x}[g(X_{S\wedge T_{n+1}})\indic_{S\geq T_1}] \notag\\
&= \Esp_{x}[g(X_{R_0\wedge T_{1}})\indic_{R_0<T_1}]+ \Esp_{x}[g(X_{S\wedge T_{n+1}})\indic_{R_0\geq T_1}] \notag\\
&=\Esp_{x}[g(X_{R_0})\indic_{R_0<T_1}]+ \Esp_{x}[\Esp_{x}[g(X_{(T_1+\theta(T_1)(S'))\wedge T_{n+1}})\vert \mathcal{F}_{1}]\indic_{R_0\geq T_1}]\notag\\
&=\Esp_{x}[g(X_{R_0})\indic_{R_0<T_1}]+ \Esp_{x}[\Esp_{Z_1}[g(X_{ S'\wedge T_{n}})]\indic_{R_0\geq T_1}] \label{demo:eq:st}\\
&\leq \Esp_{x}[g(X_{R_0})\indic_{R_0<T_1}]+ \Esp_{x}[ \Vmesure{n}(Z_1) \indic_{R_0\geq T_1}]\label{demo:eq:hyp} \\
&\leq \sup_{r\geq 0} \Esp_{x}[g(X_{r})\indic_{r<T_1}]+ \Esp_{x}[ \Vmesure{n}(Z_1) \indic_{r\geq T_1}]\notag\\
&\leq \sup_{r\geq 0} \JJmesure (\Vmesure{n},g)(x,r) \lor \KKmesure \Vmesure{n}(x) =\Vmesure{n+1}(x). \notag
\end{align} 
Line \eqref{demo:eq:st} is obtained thanks to the strong Markov property. The induction hypothesis is applied in line \eqref{demo:eq:hyp}.  This achieves the induction for \eqref{gu:aa}.\\

Let us now prove  \eqref{gu:b}. One has
\ali{\Esp_{x}[g(X_{S_{n+1,2\epsilon}})] &=\Esp_{x}[g(X_{S_{n+1,2\epsilon}})\indic_{R_{n,0}^\epsilon<T_1}]+\Esp_{x}[g(X_{S_{n+1,2\epsilon}})\indic_{R_{n,0}^\epsilon\geq T_1}] \\
&=\Esp_{x}[g(X_{R_{n,0}^\epsilon})\indic_{R_{n,0}^\epsilon<T_1}]+\Esp_{x}[g(X_{T_1+\theta(T_1)S_{n,\epsilon}})\indic_{R_{n,0}^\epsilon\geq T_1}] \\
&=\Esp_{x}[g(X_{R_{n,0}^\epsilon})\indic_{R_{n,0}^\epsilon<T_1}]+\Esp_{x}[\Esp_{Z_1}[g(X_{S_{n,\epsilon}})]\indic_{R_{n,0}^\epsilon\geq T_1}] \\
&\geq \Esp_{x}[g(X_{R_{n,0}^\epsilon})\indic_{R_{n,0}^\epsilon<T_1}]+ \Esp_{x}[\Vmesure n (Z_1)\indic_{R_{n,0}^\epsilon\geq T_1}]-\epsilon \times \proba_{x}(R_{n,0}^\epsilon\geq T_1) \\
&\geq\Esp_{x}[g(X_{R_{n,0}^\epsilon})\indic_{R_{n,0}^\epsilon<T_1}]+ \Esp_{x}[\Vmesure n (Z_1)\indic_{R_{n,0}^\epsilon\geq T_1}]-\epsilon.}
From the definition of $R_{n,\epsilon}$ one readily obtains\\
$\bullet$  if $ \KKmesure \Vmesure n (x) > \sup_{t\geq 0}\JJmesure(\Vmesure n,g)(x,t) $, then $\Vmesure{n+1}(x)=\KKmesure \Vmesure n (x)=\Esp_{x}[\Vmesure n (Z_1)]$ and $R_{n,0}^\epsilon=t^*(x)$. So  the previous inequality becomes
\begin{align*}
\Esp_{x}[g(X_{S_{n+1,2\epsilon}})]&\geq \Esp_{x}[g(X_{t^*(x)})\indic_{t^*(x)<T_1}]+ \Esp_{x}[\Vmesure n (Z_1)\indic_{T_1\leq t^*(x)}]-\epsilon\\
 &= \KKmesure \Vmesure n (x) - \epsilon= \Vmesure{n+1}(x)-\epsilon. 
\end{align*}
$\bullet$ Otherwise, $\Vmesure{n+1}(x)=\sup_{t\geq 0}\JJmesure(\Vmesure n,g)(x,t)$, and the previous inequality becomes
\begin{align*}
\Esp_{x}[g(X_{S_{n+1,2\epsilon}})] & \geq \JJmesure(\Vmesure n,g)(x,r_{0,\varepsilon}(x)) - \varepsilon\\
&\geq  \sup_{t\geq 0}\JJmesure(\Vmesure n,g)(x,t) - \varepsilon - \varepsilon = \Vmesure{n+1}(x)- 2\varepsilon.
\end{align*}
%
Thanks to these two cases, we prove \eqref{gu:b} and we end the proof.
\end{proof}
%
\section{Comparaison between the tagged cell and measure-valued process} \label{sec:compa}
In this section, we investigate wether stopping a single well chosen individual is equivalent to stopping the whole population.
When dealing with a branching population, some important characteristics of the global population, e.g. laws of large numbers for functionals of the individuals, can be obtained by simply studying a suitably weighted random tagged lineage, by means of many-to-one formulas, see e.g. \cite{Ber06, Cl17,DHKR15,Guy07}. 
We prove that this property does not hold true for the optimal stopping problem. We use the simple toy example of cell division from Section \ref{ex} to show that stopping a suitably chosen tagged cell and the whole population yield different value functions.
%
\subsection{Tagged cell and many-to-one formula}
Let us rapidly describe the many-to-one formula and the definition of the tagged as presented in \cite{DHKR15}. Heuristically, picking a branch uniformly at random along the genealogical tree describe the path of a tagged cell whose the behaviour is similar to the one of an individual picked at time $t$.

 More precisely, for any measurable positive function $h$ and any $t\geq 0$ we have the following formula, commonly called \emph{many-to-one} formula
\begin{equation} \label{many-to-one} \Esp_{x}\left[\sum_{u\in\mathcal{U}} X_t(h) \Exp^{-r t} \right]=\sum_{i=1}^n\Esp_{x_i}\left[\frac{h(\traitmarque_t)}{\traitmarque_t} \right]\times x_i, 
\end{equation}
where $x=\mu= \sum_{i=1}^n \delta_{(x_i,0)}$, $(X_t)_{t\geq 0}$ is the measure-valued PDMP of Section~\ref{sect:extous}, and $(\traitmarque_t)_{t\geq 0}$ is the tagged process, representing the evolution of the size of the tagged cell. Its dynamics is as a real-valued PDMP described in Section~\ref{sect:ex1} but whose parameters are different from those of $(X_t)_{t\geq 0}$; see \cite{ Ber06,Cl17,DHKR15} for details. 

To compare the value functions of the tagged cell and the measure-valued process, we will impose the form of our reward functions based on Equation~\ref{many-to-one}. For the tagged cell, we choose a bounded nonnegative function $f$ continuous along the flow. For the process $(X_t)_t$, the reward function $g$ has the form
\begin{equation} 
g: \sum_{i=1}^n \delta_{(x_i,t)}  \in E \mapsto \sum_{i=1}^n f(x_i) \Exp^{-rt}, 
\label{eq:reward}
 \end{equation}
where we used the time-augmented process defined in Section \ref{sec:aug} to take the time dependence into account.

\subsection{Comparison of the value functions}
Numerically computing the value function $\mathbb V_{N}$ is very demanding, as for each iteration of the dynamic programming operators, one needs to compute the functions on the whole state space $E$, at least at first sight. Actually, when dealing with the optimal stopping problem with horizon $N$, the dynamic programming recursion on functions $\mathbb V_{n}$ can be rewritten as a recursion on the random variables $\mathbb V_{k}(Z_{N-k})$. Still, the recursion is numerically intractable as one needs to compute conditional expectations at each step. In order to avoid such intricacies, we simply consider the optimal stopping problem with horizon $N=1$ jump.

We thus have to compute
\[ \begin{aligned}
 \Vmesure{0}(Z_1)&=g(Z_1), \\
\Vmesure{1}(Z_0)&=\sup_{t\geq 0} \left\{\Esp_{Z_0}[\Vmesure{0}(Z_1)\indic_{T_1 \leq t}]+g\circ\Phi(Z_{0},t)\times \proba_{Z_0}(T_1>t) \right\} \lor\Esp_{Z_0}[\Vmesure{0}(Z_1)].
\end{aligned} \]
%
The only quantities to look at are $\Vmesure{0}(Z_1)=g(Z_1)$ and $\Vmesure{1}(Z_0)=\Vmesure{1}(\mu)$. Moreover, this last expression is deterministic if we choose a deterministic $Z_0$. More specifically, one has
\begin{equation}  \Vmesure{0}(Z_1)=\Exp^{-rT_1}\left[\sum_{i \neq I_1}(x_i\Exp^{r T_1}-\gamma)\indic_{x_i<\gamma\Exp^{-r T_1}}+(x_{I_1}\Exp^{rT_1}-2\gamma)\indic_{x_{I_1}<2\gamma\Exp^{-rT_1}}+(n+1)\gamma \right], \label{eq:V0}
 \end{equation}
 and
\begin{equation}\begin{aligned} &\Vmesure{1}(Z_0)=\\
&\sup_{t>0}\left\{ \Esp_{x}\left[\Exp^{-rt}\left(\sum_{i \neq I_1}(x_i\Exp^{r T_1}-\gamma)\indic_{x_i<\gamma\Exp^{-r T_1}}+(x_{I_1}\Exp^{rT_1}-2\gamma)\indic_{x_{I_1}<2\gamma\Exp^{-rT_1}}+(n+1)\gamma \right)\indic_{T_1\leq t}\right] \right. \\
&+\left. \Exp^{-rt}\left(\sum_{i}(x_i\Exp^{r t}-\gamma)\indic_{x_i<\gamma\Exp^{-r t}}+n\gamma \right)\exp\left(-\sum(x_i^\alpha)\frac{\Exp^{rt\alpha}-1}{\alpha r} \right) \right\}\lor\Esp_{x}[\Vmesure{0}(Z_1)],\end{aligned}\label{eq:V1}\end{equation}
where the first jump time $T_1$ is distributed as \eqref{eq:loiT1} and the random variable $I_1$ is the rank of the split cell. Its distribution is given by $\proba_{\mu}(I_1=j)=\frac{x_j^\alpha}{x_1^\alpha+\ldots+x_n^\alpha}$. We can numerically simulate $\Vmesure{0}(Z_1)$. For $\Vmesure{1}(\mu)$, we fix a (large enough) maximum time and discretize the time interval in order to compute (an approximation of) the supremum. 
In the same way, the value function for the tagged cell is given by:
 \ali{ \Vmarque{1}(x)&=\sup_{t\leq 0}\left\{ \left[(x\Exp^{rt}-\gamma)\indic_{x\Exp^{rt}<\gamma}+\gamma\right]\exp\left(-\frac{x^{\alpha}}{\alpha r}(\Exp^{r\alpha t}-1)\right) \right.\\
&+\left. \Esp_{x}\left[ \left( (x\Exp^{r T_1}/2-\gamma) \indic_{x\Exp^{r T_1}<2\gamma}+\gamma \right)\indic_{T_1 \leq t} \right]  \right\} \lor \Esp_x\left[\Vmarque{0}(\Zmarque_1) \right].}

For the numerical results, we use the following parameters: $\alpha=1$, $r=2$, $\gamma=1$. For the discretization step, we evaluate the maximum value of $T_1$ by the Monte-Carlo simulations and we divide it by $nbpt$ which corresponds to the number of discretization points for the evaluation of the supremum. Here, $nbpt=10000$ and the Monte Carlo sample size is $N=100000$. We obtain the following approximation 
\begin{equation}
\begin{aligned}
 \Vmarque{1}(3)=1.0018 \pm 2.54\times10^{-4} \\
 \Vmesure{1}(\delta_3)=  1.3447 \pm 3.6034\times10^{-4} .
 \end{aligned}
\end{equation}
Hence, even on this simple toy example where the cost functions $g$ and $h$ match and with a short one-step horizon, the optimal performance for the global population and the tagged cell differ.
It will be the object of future work to design specific numerical approximations of the value function of the global measure-valued population.



\begin{thebibliography}{10}

\bibitem{bauerle11}
{\sc B{\"a}uerle, N., and Rieder, U.}
\newblock {\em Markov decision processes with applications to finance}.
\newblock Universitext. Springer, Heidelberg, 2011.

\bibitem{BL16}
{\sc Benaïm, M., and Lobry, C.}
\newblock Lotka–volterra with randomly fluctuating environments or “how
  switching between beneficial environments can make survival harder”.
\newblock {\em Ann. Appl. Probab. 26}, 6 (12 2016), 3754--3785.

\bibitem{Ber06}
{\sc Bertoin, J.}
\newblock {\em Random fragmentation and coagulation processes}, vol.~102 of
  {\em Cambridge Studies in Advanced Mathematics}.
\newblock Cambridge University Press, Cambridge, 2006.

\bibitem{BR11}
{\sc Buckwar, E., and Riedler, M.~G.}
\newblock An exact stochastic hybrid model of excitable membranes including
  spatio-temporal evolution.
\newblock {\em J. Math. Biol. 63}, 6 (2011), 1051--1093.

\bibitem{CCF16}
{\sc Campillo, F., Champagnat, N., and Fritsch, C.}
\newblock Links between deterministic and stochastic approaches for invasion in
  growth-fragmentation-death models.
\newblock {\em J. Math. Biol. 73}, 6-7 (2016), 1781--1821.

\bibitem{Cl17}
{\sc Cloez, B.}
\newblock Limit theorems for some branching measure-valued processes.
\newblock {\em Advances in Applied Probability 49}, 2 (2017), 549–580.

\bibitem{CDGMMY}
{\sc Cloez, B., Dessalles, R., Genadot, A., Malrieu, F., Marguet, A., and
  Yvinec, R.}
\newblock Probabilistic and piecewise deterministic models in biology.
\newblock {\em ESAIM: Procs 60\/} (2017), 225--245.

\bibitem{Cos16}
{\sc Costa, M.}
\newblock A piecewise deterministic model for a prey-predator community.
\newblock {\em Ann. Appl. Probab. 26}, 6 (2016), 3491--3530.

\bibitem{davis84}
{\sc Davis, M. H.~A.}
\newblock Piecewise-deterministic {M}arkov processes: a general class of
  nondiffusion stochastic models.
\newblock {\em J. Roy. Statist. Soc. Ser. B 46}, 3 (1984), 353--388.
\newblock With discussion.

\bibitem{davis93}
{\sc Davis, M. H.~A.}
\newblock {\em Markov models and optimization}, vol.~49 of {\em Monographs on
  Statistics and Applied Probability}.
\newblock Chapman \& Hall, London, 1993.

\bibitem{PDP15}
{\sc de~{S}aporta, B., Dufour, F., and {Z}hang, H.}
\newblock {\em Numerical methods for simulation and optimization of piecewise
  deterministic {M}arkov processes: application to reliability}.
\newblock Mathematics and statistics series. Wiley-ISTE, 2015.
\newblock hal-01249897.

\bibitem{DHKR15}
{\sc Doumic, M., Hoffmann, M., Krell, N., and Robert, L.}
\newblock Statistical estimation of a growth-fragmentation model observed on a
  genealogical tree.
\newblock {\em Bernoulli 21}, 3 (2015), 1760--1799.

\bibitem{FM04}
{\sc Fournier, N., and Méléard, S.}
\newblock A microscopic probabilistic description of a locally regulated
  population and macroscopic approximations.
\newblock {\em Ann. Appl. Probab. 14}, 4 (11 2004), 1880--1919.

\bibitem{GT14}
{\sc Genadot, A., and Thieullen, M.}
\newblock Multiscale piecewise deterministic {M}arkov process in infinite
  dimension: central limit theorem and {L}angevin approximation.
\newblock {\em ESAIM Probab. Stat. 18\/} (2014), 541--569.

\bibitem{gugerli}
{\sc Gugerli, U.~S.}
\newblock Optimal stopping of a piecewise-deterministic {M}arkov process.
\newblock {\em Stochastics 19}, 4 (1986), 221--236.

\bibitem{Guy07}
{\sc Guyon, J.}
\newblock Limit theorems for bifurcating {M}arkov chains. {A}pplication to the
  detection of cellular aging.
\newblock {\em Ann. Appl. Probab. 17}, 5-6 (2007), 1538--1569.

\bibitem{Ka86}
{\sc Kallenberg, O.}
\newblock {\em Random measures, theory and applications}, vol.~77 of {\em
  Probability Theory and Stochastic Modelling}.
\newblock Springer, Cham, 2017.

\bibitem{Ko11}
{\sc Kolokoltsov, V.~N.}
\newblock {\em Markov processes, semigroups and generators}, vol.~38 of {\em De
  Gruyter Studies in Mathematics}.
\newblock Walter de Gruyter \& Co., Berlin, 2011.

\bibitem{RTW12}
{\sc Riedler, M.~G., Thieullen, M., and Wainrib, G.}
\newblock Limit theorems for infinite-dimensional piecewise deterministic
  {M}arkov processes. {A}pplications to stochastic excitable membrane models.
\newblock {\em Electron. J. Probab. 17\/} (2012), no. 55, 48.

\bibitem{krell}
{\sc Robert, L., Hoffmann, M., Krell, N., Aymerich, S., Robert, J., and Doumic,
  M.}
\newblock Division in escherichia coliis triggered by a size-sensing rather
  than a timing mechanism.
\newblock {\em BMC Biology 12}, 1 (Feb 2014), 17.

\end{thebibliography}
\end{document}